\documentclass[a4paper,10pt]{article}

\usepackage{graphicx}
\usepackage{geometry}
\usepackage{amsthm}
\usepackage{amssymb}
\usepackage{amsmath}
\usepackage{hyperref}
\usepackage{xcolor}
\usepackage{enumerate}
\usepackage{listings}
\usepackage{complexity}
\usepackage{blkarray}

\usepackage[font=small,labelfont=bf]{caption}
\usepackage[font=small,labelfont=normalfont,labelformat=simple]{subcaption}

\newcommand{\cev}[1]{\reflectbox{\ensuremath{\vec{\reflectbox{\ensuremath{#1}}}}}}

\graphicspath{{figs/}}

\newtheorem{theorem}{Theorem}

\newtheorem{lemma}[theorem]{Lemma}
\newtheorem{corollary}[theorem]{Corollary}
\newtheorem{question}{Question}

\newtheorem{proposition}{Proposition}
\newtheorem{conjecture}{Conjecture}

\author
{
Raphael Steiner \\
Institut f\"{u}r Mathematik \\
TU Berlin \\
\texttt{steiner@math.tu-berlin.de}
}

\date{\today}

\title{A Note on Graphs of Dichromatic Number $2$}

\begin{document}
\maketitle

\begin{abstract}
Neumann-Lara and \u{S}krekovski conjectured that every planar digraph is $2$-colourable. We show that this conjecture is equivalent to the more general statement that all oriented $K_5$-minor-free graphs are $2$-colourable.
\end{abstract}

\section{Introduction}

Digraphs and graphs considered here are loopless, and without parallel or anti-parallel arcs. A directed edge starting in $u$ and ending in $v$ is denoted by $(u,v)$, $u$ is called its \emph{tail} while $v$ is its \emph{head}. In a digraph $D$, a vertex set $X \subseteq V(D)$ is called \emph{acyclic} if the induced subdigraph $D[X]$ is acyclic. An \emph{acyclic colouring} of $D$ with $k$ colours is a mapping $c:V(D) \rightarrow [k]$ such hat $c^{-1}(\{i\})$ is acyclic for all $i \in [k]$. The \emph{dichromatic number} $\vec{\chi}(D)$ is defined as the minimal $k \ge 1$ for which such a colouring exists. For an undirected graph $G$, the dichromatic number $\vec{\chi}(G)$ is defined as the maximum dichromatic number an orientation of $G$ can have.

This notion has been introduced in 1982 by Neumann-Lara \cite{ErdNeuLara}, was rediscovered by Mohar \cite{MoharEdgeWeight}, and since then has received further attention, see \cite{perfect, fractionalNL, largesubdivisions, dig5, dig4, lists, HARUTYUNYAN2019} for some recent results. 

In analogy to the famous Four-Colour-Theorem, the following intriguing conjecture was proposed by Neumann-Lara \cite{neulara85} and independently by \u{S}krekovski \cite{bokal2004circular}.
\begin{conjecture}
If $G$ is a planar graph, then $\vec{\chi}(G) \leq 2$. 
\end{conjecture}

The strongest partial result obtained so far is due to Li and Mohar who showed the following:
\begin{theorem}[\cite{dig4}]
Every planar digraph without directed triangles admits an acyclic $2$-colouring.
\end{theorem}

The purpose of this note is to show the following.

\begin{theorem} \label{thm:main}
The following statements are equivalent:
\begin{itemize}
\item Every planar graph $G$ has $\vec{\chi}(G) \leq 2$.
\item Every $K_5$-minor-free graph $G$ fulfils $\vec{\chi}(G) \leq 2$. Moreover, any orientation of $G$ admits an acyclic $2$-colouring without monochromatic triangles.
\end{itemize}
\end{theorem}

This strengthening is similar to the situation for undirected graph colourings, where it is known that all $K_5$-minor-free graphs are $4$-colourable \cite{Wagner1937}. 

\section{$2$-Colourings of Planar Digraphs}
In the following, we will use the term \emph{planar triangulation} when we mean a maximal planar graph on at least three vertices. It is well-known that the latter (up to the choice of the outer face and reflections) admit combinatorially unique crossing-free embeddings in the plane or on the sphere, in which every face is bounded by a triangle (from now on called facial triangles). A frequent tool in our proof will be the following Lemma, which has already been used in \cite{dig4}.
\begin{lemma}[\cite{dig4}]\label{cliques}
Let $D_1$ and $D_2$ be digraphs which intersect in a tournament. Suppose that $c_1:V(D_1) \rightarrow [k]$, $c_2:V(D_2) \rightarrow [k]$ are acyclic colourings such that $c_1|_{V(D_1) \cap V(D_2)}=c_2|_{V(D_1) \cap V(D_2)}$. Then the common extension of $c_1$ and $c_2$ to $V(D_1) \cup V(D_2)$ defines an acyclic $k$-colouring of $D_1 \cup D_2$. 
\end{lemma}
In this section we prepare the proof of Theorem \ref{thm:main} with some strengthend but equivalent formulations of Neumann-Lara's Conjecture.
\begin{proposition}
The following statements are equivalent:
\begin{enumerate}[\label=(i)]
\item Neumann-Lara's Conjecture, i.e., every planar digraph has an acyclic $2$-colouring.
\item Every oriented planar triangulation admits an acyclic $2$-colouring without monochromatic facial triangles.
\item For any planar triangulation $T$, any facial triangle $a_1a_2a_3$ in $T$, and any non-monochromatic pre-colouring $p:\{a_1,a_2,a_3\} \rightarrow \{1,2\}$, every orientation $\vec{T}$ of $T$ admits an acyclic $2$-colouring $c:V(\vec{T}) \rightarrow \{1,2\}$ without monochromatic facial triangles such that $c(a_i)=p(a_i),$ $i \in \{1,2,3\}$.
\item For any planar triangulation $T$, any triangle $a_1a_2a_3$ in $T$, and any non-monochromatic pre-colouring $p:\{a_1,a_2,a_3\} \rightarrow \{1,2\}$, every orientation $\vec{T}$ of $T$ admits an acyclic $2$-colouring $c:V(\vec{T}) \rightarrow \{1,2\}$ without monochromatic triangles such that $c(a_i)=p(a_i), i \in \{1,2,3\}$.
\end{enumerate}
\begin{figure}[h]
\centering
\includegraphics[scale=1]{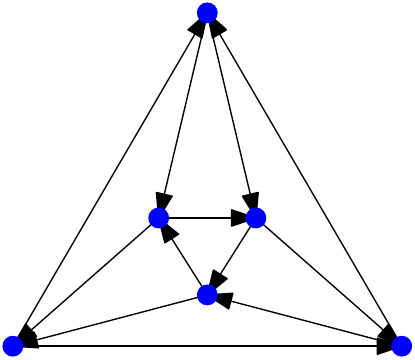} 
\caption{The octahedron-orientation $\mathcal{O}_6$.} \label{octahedron}
\end{figure}
\begin{figure}[h]
\centering
\includegraphics[scale=0.8]{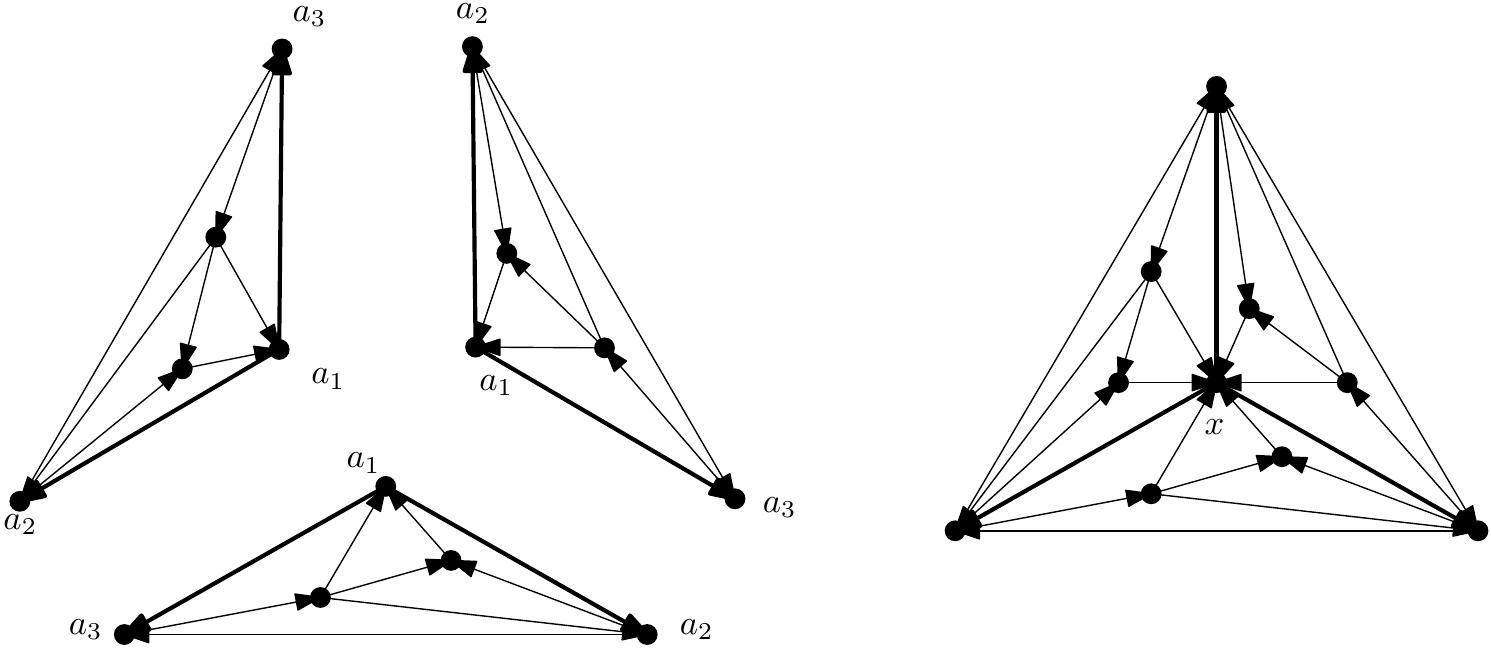}
\vspace{5pt}
\includegraphics[scale=0.8]{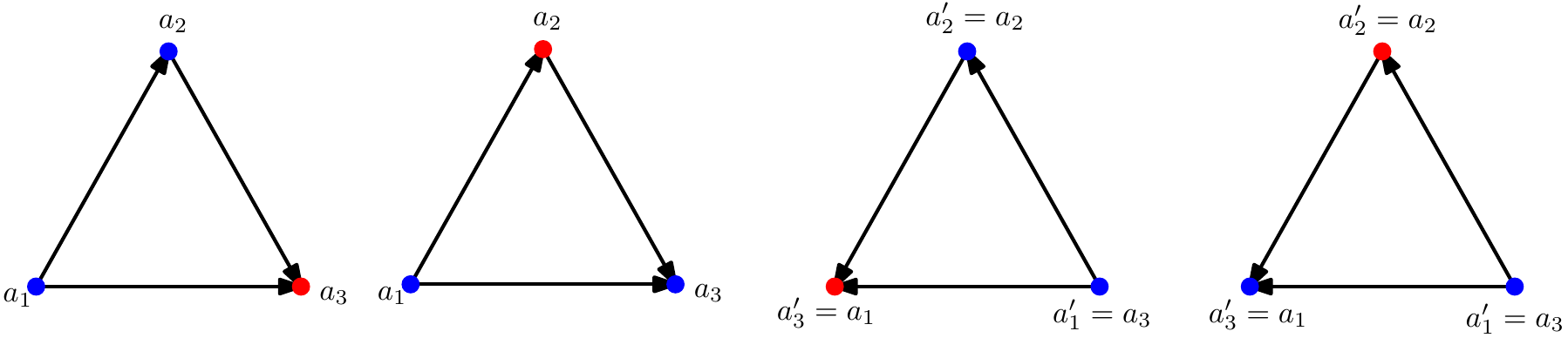}
\caption{Top: Illustration of the construction of $\vec{T}^\ast$ from three copies of $\vec{T}$. Bottom: The restriction of the colourings $c_1,c_2,c_1',c_2'$ to the triangle $t$. Blue vertices have colour $1$, red corresponds to colour $2$.} \label{construction}
\end{figure}
\begin{proof} ~
{$(i) \Longrightarrow (ii)$:} Suppose that every planar digraph is $2$-colourable and let $\vec{T}$ be an arbitrary orientation of a planar triangulation $T$. Looking at the orientation $\mathcal{O}_6$ of the octahedron graph depicted in Figure \ref{octahedron}, it is easily observed that in any acyclic $2$-colouring the triangle bounding the outer face cannot be monochromatic. Now consider a crossing-free spherical embedding of $\vec{T}$. For every facial triangle in this embedding whose orientation is transitive, we take a copy of $\mathcal{O}_6$ and glue this copy into the face in such a way that the outer three edges of $\mathcal{O}_6$ are identified with the three edges of the facial triangle (to make the orientations of the identified edges compatible, it might be necessary to reflect and rotate the embedding of $\mathcal{O}_6$ shown in Figure \ref{octahedron}). This creates a crossing-free embedding of a planar oriented triangulation $\vec{T}^\Delta$. By assumption, $\vec{T}^\Delta$ admits an acyclic $2$-colouring. This colouring restricted to the vertices of the subdigraph $\vec{T}$ clearly is still valid. Furthermore, no triangle in $\vec{T}$ can be monochromatic: This follows by definition if the triangle forms a directed cycle. If the orientation is transitive, by definition of $\vec{T}^\Delta$, a monochromatic colouring would contradict the fact that $\mathcal{O}_6$ has no acyclic $2$-colouring with the outer three vertices being coloured the same. 
\paragraph{$(ii) \Longrightarrow (iii)$:} Suppose that $(ii)$ holds. Let $\vec{T}$ be an orientation of a planar triangulation $T$ and let $a_1a_2a_3$ be the vertices of a facial triangle $t$ of $T$, equipped with a non-monochromatic pre-colouring $p:\{a_1,a_2,a_3\} \rightarrow \{1,2\}$. 

\paragraph{Case 1: $t$ is not directed in $\vec{T}$.}

By relabelling, we may assume the transitive orientation $(a_1,a_2), (a_1,a_3), (a_2,a_3) \in E(\vec{T})$. Now consider a plane embedding of $\vec{T}$ in which $a_1a_2a_3$ forms the bounding triangle of the outer face and where $a_1,a_2,a_3$ appear in clockwise order. We now define a new oriented planar triangulation $\vec{T}^\ast$ as follows: We consider the embedding and orientation of the $K_4$ as shown in Figure \ref{construction}, in which the outer face has a clockwise direction and the central vertex is a source. Into each of the three inner faces of this embedding, we can now glue a copy of $\vec{T}$ with the described embedding in such a way that all orientations of identified edges agree. The vertex $a_1$ from each copy now is identified with the central vertex, which we call $x$. This oriented planar triangulation $\vec{T}^\ast$ according to assumption admits an acyclic $2$-colouring $c^\ast:V(\vec{T}^\ast) \rightarrow \{1,2\}$ without monochromatic facial triangles. By relabelling the colours, we may assume that $c^\ast(x)=1$. Because the outer triangle is not monochromatic, there have to be edges $e_1=(u_1,v_1)$ and $e_2=(u_2,v_2)$ on the outer triangle such that $c^\ast(u_1)=c^\ast(v_2)=1, c^\ast(u_2)=c^\ast(v_1)=2$. To both $e_1, e_2$ we have a corresponding copy of $\vec{T}$, and the $2$-colourings $c_1,c_2$ of $\vec{T}$ which $c^\ast$ induces on these copies are still valid acyclic colourings. Furthermore, there are no monochromatic facial triangles in $\vec{T}$ with respect to $c_1,c_2$: Every bounded facial triangle of $\vec{T}$ in the corresponding copy also forms a bounded facial triangle of $\vec{T}^\ast$, while the outer triangles of the copies contain $e_1$ respectively $e_2$ and are thus not monochromatic under $c_1$ respectively $c_2$. From the way we glued the three copies of $\vec{T}$ we conclude that $c_1(a_1)=c^\ast(x)=1, c_1(a_2)=c^\ast(u_1)=1, c_1(a_3)=c^\ast(v_1)=2$ and $c_2(a_1)=c^\ast(x)=1, c_2(a_2)=c^\ast(u_2)=2, c_2(a_3)=c^\ast(v_2)=1$. 

Now consider the oriented planar triangulation $\cev{T}$ which is obtained from $\vec{T}$ by reversing the orientations of all edges. Let us rename the vertices of $t$ according to $a_1':=a_3, a_3':=a_1, a_2':=a_2$. We have $(a_1',a_2'),(a_1',a_3'),(a_2',a_3') \in E(\cev{T})$. We can therefore apply the same arguments as above to $\cev{T}$ with the transitive labeling $a_1'a_2'a_3'$ of the triangle $t$. Hence, we obtain acyclic $2$-colourings $c_1',c_2':V(\cev{T}) \rightarrow \{1,2\}$ of $\cev{T}$ without monochromatic facial triangles such that $c_1'(a_1')=c_1'(a_2')=1, c_1'(a_3')=2$ and $c_2'(a_1')=1, c_2'(a_2')=2, c_2'(a_3')=1$. 

Finally, because the acyclic $2$-colourings of $\vec{T}$ and $\cev{T}$ coincide, we conclude that $c_1,c_2,c_1',c_2':V(\vec{T}) \rightarrow \{1,2\}$ all form acyclic $2$-colourings of $\vec{T}$ without monochromatic facial triangles. It is easy to see that the colourings $c_1,c_2,c_1'$ together extend all possible non-monochromatic pre-colourings $p$ of the triangle $t$ with exactly one vertex of colour $2$ (see the illustration in Figure \ref{construction}). Hence, after flipping all colours from $1$ to $2$ and $2$ to $1$ we have found extending colourings for all possible pre-colourings $p$ without monochromatic facial triangles, as desired.

\paragraph{Case 2: $t$ is directed in $\vec{T}$.}

After relabelling (and possibly exchanging the colours $1$ and $2$), we may suppose that $(a_1,a_2),(a_2,a_3),(a_3,a_1) \in E(\vec{T})$ and $p(a_1)=1,p(a_2)=1,p(a_3)=2$. Consider the orientation $\vec{T}_e$ obtained from $\vec{T}$ by reversing the edge $e=(a_3,a_1)$. By the first case, we know that $\vec{T}_e$ admits an acyclic $2$-colouring $c_e$ without monochromatic facial triangles which extends $p$. Because the endpoints of $e$ receive different colours, it follows directly that $c_e$ also defines an acyclic $2$-colouring of $\vec{T}$ with the required properties, and the claim follows also in this case.
\paragraph{$(iii) \Longrightarrow (iv)$:} We prove the statement by induction on the number of vertices. In the base case, where $\vec{T}$ is an oriented triangle, the statement clearly holds true. Now let $n \ge 4$, and assume that the statement holds for all triangulations with less than $n$ vertices. Let $\vec{T}$ be an arbitrary orientation of some planar triangulation $T$ with $n$ vertices. If $T$ is $4$-connected, then the only triangles in $T$ are the facial triangles and therefore the claim follows from $(iii)$. Therefore we may suppose that $T$ is not $4$-connected, i.e., there exists a \emph{separating triangle} $x_1x_2x_3$ in $T$. Consider some plane crossing-free embedding of $T$. Here, $x_1x_2x_3$ separates the vertices in its interior ($V_\text{in}  \subseteq V(T)$) from those in its exterior ($V_\text{out} \subseteq V(T)$). Let $\vec{T}_\text{out}:=\vec{T}[V_\text{out} \cup \{x_1,x_2,x_3\}]$ and $\vec{T}_\text{in}:=\vec{T}[V_\text{in} \cup \{x_1,x_2,x_3\}]$. Both form oriented planar triangulations on less than $n$ vertices. To prove that $\vec{T}$ satisfies the inductive claim, let $t=a_1a_2a_3$  be a given triangle in $T$ equipped with a non-monochromatic pre-colouring $p:\{a_1,a_2,a_3\} \rightarrow \{1,2\}$. We must either have $\{a_1,a_2,a_3\} \subseteq V_\text{out} \cup \{x_1,x_2,x_3\}$ or $\{a_1,a_2,a_3\} \subseteq V_\text{in} \cup \{x_1,x_2,x_3\}$. Assume that we are in the first case, the second case is completely analogous. Then, by the induction hypothesis, there exists an acyclic colouring $c_\text{out}:V(\vec{T}_\text{out}) \rightarrow \{1,2\}$ without monochromatic triangles such that $c(a_i)=p(a_i), i \in \{1,2,3\}$. The restriction of $c_\text{out}$ to $x_1x_2x_3$ now defines a non-monochromatic pre-colouring for $\vec{T}_\text{in}$, and it follows from the induction hypothesis that there exists an acyclic $2$-colouring $c_2$ of $\vec{T}_\text{in}$ without monochromatic triangles which agrees with $c_\text{out}$ on $V(\vec{T}_\text{out}) \cap V(\vec{T}_\text{in})=\{x_1,x_2,x_3\}$. By Lemma \ref{cliques}, the common extension of $c_\text{out}, c_\text{in}$ to $V(\vec{T})$ now defines an acyclic $2$-colouring of $\vec{T}$, extending the given pre-colouring $p$ of $t$ and without monochromatic triangles. This verifies the inductive claim.
\paragraph{$(iv) \Rightarrow (i)$:} This follows since every planar graph is a subgraph of a planar triangulation. 
\end{proof}
\end{proposition}
Because any edge in a planar triangulation lies on a triangle, we directly obtain the following.
\begin{corollary}\label{obvious}
Under the assumption of Neumann-Lara's Conjecture, every orientation of a planar graph admits an acyclic $2$-colouring without monochromatic triangles which can be chosen to extend any given pre-colouring of an edge or any non-monochromatic pre-colouring of a triangle. 
\end{corollary}
\section{$K_5$-Minor-Free Graphs} 
Given a pair $G_1, G_2$ of undirected graphs such that $V(G_1) \cap V(G_2)$ forms a clique of size $i$ in both $G_1$ and $G_2$, and such that $|V(G_1)|,|V(G_2)|>i$, the graph $G$ with $V(G)=V(G_1) \cup V(G_2)$ and $E(G)=E(G_1) \cup E(G_2)$ is called the \emph{proper $i$-sum} of $G_1$ and $G_2$. A graph obtained from $G$ by deleting some (possibly all or none) of the edges in $E(G_1) \cap E(G_2)$ is said to be an \emph{$i$-sum} of $G_1$ and $G_2$. The central tool in proving Theorem \ref{thm:main} is the following classical result due to Wagner. By $V_8$ we denote the so-called \emph{Wagner graph}, that is the graph obtained from $C_8$ by joining any two diagonally opposite vertices by an edge. 
\begin{theorem}[\cite{Wagner1937}]\label{wagner}
A simple graph is $K_5$-minor-free if and only if it can be obtained from planar graphs and copies of $V_8$ by means of repeated $i$-sums with $i \in \{0,1,2,3\}$. 
\end{theorem} 
$V_8$ is triangle-free and admits an acyclic $2$-colouring for any orientation. Moreover, it can be easily checked that such a colouring can be chosen to extend any given pre-colouring of two adjacent vertices. We are now in the position to prove Theorem \ref{thm:main}.

\begin{proof}[Proof of Theorem \ref{thm:main}.]
Assume that Neumann-Lara's Conjecture holds true. We have to prove that every oriented $K_5$-minor-free graph admits an acyclic $2$-colouring without monochromatic triangles. In fact, we prove the following slightly stronger statement:
\\
\emph{Every orientation of a $K_5$-minor-free graph admits an acyclic $2$-colouring without monochromatic triangles which can be chosen to extend any given pre-colouring of an edge or any non-monochromatic pre-colouring of a triangle. }
\\
Assume towards a contradiction that there exists a $K_5$-minor-free graph $G$ which does not satisfy this claim and choose $G$ minimal with respect to the number of vertices, and among all such graphs maximal with respect to the number of edges.

By Corollary~\ref{obvious}, we know that the claim is fulfilled by all planar graphs and by $V_8$, and so it follows from Theorem~\ref{wagner} that $G$ is the $i$-sum of two $K_5$-minor-free graphs $G_1, G_2$ with fewer vertices, where $0 \le i \le 3$. By the minimality assumption, we therefore know that $G_1, G_2$ satisfy the claim. Clearly, every super-graph of $G$ does not satisfy the assertion as well. Therefore, by the assumed edge-maximality, $G$ must in fact be the proper $i$-sum of $G_1$ and $G_2$. 

Now choose some orientation $\vec{G}$ of $G$ for which the above claim fails. Denote by $\vec{G}_1, \vec{G}_2$ the induced orientations on the subgraphs $G_1, G_2$.

Let $e=uv \in E(G)=E(G_1) \cup E(G_2)$ be an arbitrary edge with a given pre-colouring $p:\{u,v\} \rightarrow \{1,2\}$. W.l.o.g. assume that $e \in E(G_1)$. Let $c_1:V(G_1) \rightarrow \{1,2\}$ be an acyclic $2$-colouring of $\vec{G}_1$ without monochromatic triangles which extends $p$. The clique $C=V(G_1) \cap V(G_2)$ in $G_2$ is either empty, a single vertex, and edge or a triangle. In each case, the restriction $p'=c_1|_C$ (if non-empty) can be considered as a pre-colouring of a vertex, an edge or a triangle in $G_2$ with two colours. In the case where $C$ is a triangle, by the choice of $c_1$, we furthermore know that the pre-colouring $p$ is not monochromatic. We therefore conclude that in any case, there is an acyclic $2$-colouring $c_2:V(G_2) \rightarrow \{1,2\}$ of $\vec{G}_2$ without monochromatic triangles which extends $p'$. Therefore $c_1$ and $c_2$ agree on $V(G_1) \cap V(G_2)=C$ and it follows from Lemma \ref{cliques} that the common extension of $c_1, c_2$ to $V(G)$ defines an acyclic $2$-colouring of $\vec{G}$ which extends $p$. Because every triangle of $G$ is fully contained in $G_1$ or $G_2$, we also have that there are no monochromatic triangles under $c$. 

Similarly, for any triangle $t=a_1a_2a_3$ in $G$ equipped with a non-monochromatic pre-colouring $p:\{a_1,a_2,a_3\} \rightarrow \{1,2\}$, we may assume w.l.o.g. that $t$ is fully contained in $G_1$. Again, we find a pair of acyclic $2$-colourings $c_1,c_2$ of $\vec{G}_1, \vec{G}_2$ such that $c_1$ extends $p$, $c_2$ coincides with $c_1$ on the clique $V(G_1) \cap V(G_2)$, and there are no monochromatic triangles in $G_j$ under $c_j$ for $j=1,2$. Finally, the common extension of $c_1,c_2$ to $V(G)$ by Lemma \ref{cliques} defines an acyclic $2$-colouring of $\vec{G}$ with the desired properties.

From this we conclude that $\vec{G}$ admits an extending acyclic $2$-colouring without monochromatic triangles for any pre-colouring of an edge and for any non-monochromatic pre-colouring of a triangle. This is a contradiction to our choice of $\vec{G}$. This shows that the initial assumption was false and concludes the proof of the Theorem.
\end{proof} 
\section{Conclusion}
A natural question that comes out from the discussion in this paper is the following.
\begin{question}
What is the largest minor-closed class $\mathcal{G}_2$ of undirected graphs with dichromatic number at most $2$?
\end{question}
While $\vec{\chi}(K_6)=2$, it is known that $\vec{\chi}(K_7)=3$. Therefore, $\mathcal{G}_2$ is a subclass of the $K_7$-minor-free graphs. However, $\mathcal{G}_2$ seems to be a lot smaller than this class. In fact, there are examples of $K_{3,3}$-minor-free graphs with dichromatic number greater than $2$, see Figure \ref{k33} for a simple example of such a graph.
\begin{figure}[h]
\centering
\includegraphics[scale=1]{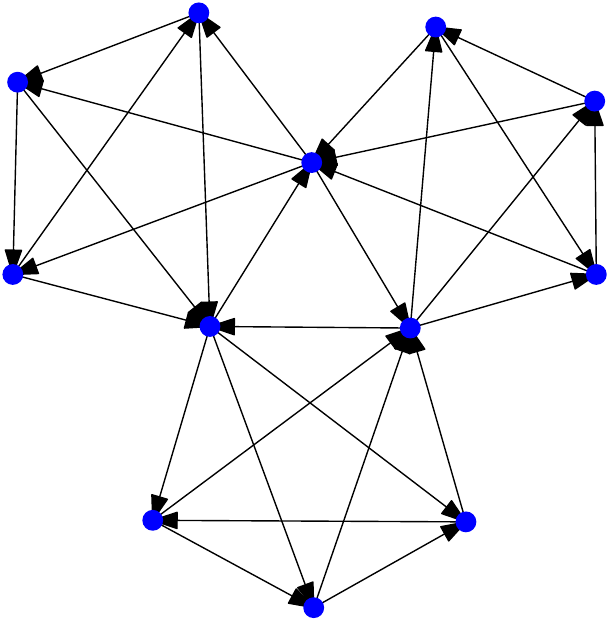}
\caption{An oriented $K_{3,3}$-minor-free graph without an acyclic $2$-colouring.} \label{k33}
\end{figure}
\bibliography{references}
\bibliographystyle{alpha}

\end{document}